%% file: ArXiv.tex
\documentclass[reqno]{amsart}
\usepackage[T1]{fontenc}
\usepackage[pdftex]{color,graphicx}
\usepackage{wrapfig} \usepackage[numbers,square]{natbib}
\usepackage{caption}
\title[Singularities in EHD equations]{Singularities in axisymmetric free boundaries for ElectroHydroDynamic equations}

\author[M. Smit Vega Garcia]{Mariana Smit Vega Garcia}
\address{Department of Mathematics, University of Duisburg-Essen, Thea-Leymann-Strasse 9, 45127, Essen, Germany}
\email{mariana.vega-smit@uni-due.de}
\author[E. Varvaruca]{Eugen Varvaruca}
\address{Faculty of Mathematics, ``Al. I. Cuza" University, Bdul. Carol I, nr. 11, 700506, Ia\c{s}i, Romania}
\email{eugen.varvaruca@uaic.ro}
\author[G.S. Weiss]{Georg S. Weiss}
\address{Department of Mathematics, University of Duisburg-Essen, Thea-Leymann-Strasse 9, 45127, Essen, Germany}
\email{georg.weiss@uni-due.de}
\thanks{M. Smit Vega Garcia and G.S. Weiss have been partially supported by the project
``Singularities in ElectroHydroDynamic equations'' of the German Research Council}
\date{}

\theoremstyle{plain}
\newtheorem{theorem}{Theorem}[section]

\newtheorem{lemma}[theorem]{Lemma}
\newtheorem{proposition}[theorem]{Proposition}
\newtheorem{corollary}[theorem]{Corollary}

\theoremstyle{definition}
\newtheorem{definition}[theorem]{Definition}
\theoremstyle{definition}
\newtheorem{remark}[theorem]{Remark}
\numberwithin{equation}{section}

\def\R{{\bf R}}

\def\N{{\bf N}}

\def\div{\textrm{\rm div}}

\newcommand{\D}{\partial}
\newcommand{\Om}{\Omega}

\newcommand{\non}{\nonumber}
\newcommand{\be}{\begin{equation}}
\newcommand{\ee}{\end{equation}}

\begin{document}

\begin{abstract}
We consider singularities in the ElectroHydroDynamic equations. In a regime where we are allowed to neglect surface tension, and assuming that the free surface is given by an injective curve and that either the fluid velocity or the electric field satisfies a certain non-degeneracy condition, we prove that either the fluid region or the gas region is asymptotically a {\em cusp}.
Our proofs depend on a combination of monotonicity formulas and a non-vanishing result by Caffarelli and Friedman.
As a by-product of our analysis we also obtain a special solution with convex conical air-phase which we believe to be new.
\end{abstract}

\maketitle

\section{Introduction}\label{S:intro}

In his pioneering paper of 1964, \cite[]{taylor}, Sir Geoffrey Taylor describes an experiment for the formation of a
liquid cone by exposing a fluid jet to an electric field, and he formally derives a formula for the electric potential (referred to as the {\em Taylor-cone solution} in the sequel) of that field. At a critical voltage value, the surface of the fluid ruptures, and a fluid jet forms, a phenomenon which has found applications as various as electrospraying, electrospinning and, to give more concrete examples, ink jet printers, mass spectrometers and the production of lab-on-the-chips. Despite these various applications and extensive research from a physical point of view, the phenomenon of the Taylor cone as well as other singularities in electro-hydrodynamics remain in effect untouched by mathematical analysis, possibly due to the difficulty of the free boundary problem arising from the particular ElectroHydroDynamic equations (EHD equations) used as a model.

We use as a basis for our analysis the simplest model available. We consider a stationary, irrotational flow of an incompressible, inviscid, perfectly conducting liquid; outside the fluid there is a dielectric gas, and the stationary electric field is driven by the potential difference between the perfectly conducting liquid and some fixed outer domain boundary or infinity. Motivated by particular singularities on which gravity is supposed to have no influence (this point will be underlined by the heuristic argument below) we will neglect the influence of gravity. Since the fluid is a perfect conductor, the stationary electric potential $\phi$ is constant in the fluid region, and we may assume it to have the value $0$ there. In this setting, the ElectroHydroDynamic equations simplify
(see \cite[(10)]{vdb} as well as \cite[Section 2]{zubarev})
to
\begin{align}
&\Delta \phi = 0 \textrm{ in the gas region,}\label{eins}\\
&\Delta V =0 \textrm{ in the fluid region,}\\
&|\nabla \phi|^2 - |\nabla V|^2 = \kappa+B\textrm{ on the free surface of the fluid,}\label{bern}\\
& \phi=0\textrm{ on the free surface of the fluid,}\\
&\langle \nabla V,\nu\rangle = 0 \textrm{ on the free surface of the fluid,}\label{vier}
\end{align}
where $B$ is a constant, $V$ is the velocity potential of the stationary fluid, $\nu$ is the outward pointing
unit normal and $\kappa$ the mean curvature on the boundary of the fluid phase. Note that we choose the sign of
the mean curvature of the boundary of the set $A$ so that $\kappa$ is positive on convex portions of $\partial A$.

Viewed as a free boundary problem, problem (\ref{eins})-(\ref{vier}) is new, so there are no results
from that perspective.
Possible reasons for the lack of results may be the "bad" sign of the mean curvature (explained in more
detail below) as well as the Neumann boundary condition (\ref{vier}). While there are many
results concerning free boundaries that are level sets, there are relatively few results on problems without this property.

There are other related free boundary/free discontinuity problems we should mention.
For example, in \cite{ACS}, the authors study the free boundary problem
\begin{align}
&\Delta u = 0 \textrm{ in }\Omega\cap(\{ u>0\}\cup \{ u<0\}),\label{acseq1}\\
& |\nabla u^+|^2 - |\nabla u^-|^2 = -\kappa \textrm{ on the free surface } \partial\{ u\le 0\}\cap\Omega,\label{acseq2}
\end{align}
where $\Om$ is a smooth domain of $\R^n$. However, even in the case $u^-\equiv 0$, problem (\ref{acseq1})-(\ref{acseq2}) differs
from (\ref{eins})-(\ref{vier}) by the sign of the mean curvature. This becomes clearer
when comparing the energy functionals associated to the two problems: in the case of one-phase solutions ($u^-\equiv 0$) of (\ref{acseq1})-(\ref{acseq2}), the energy takes the form
$$
P_\Omega(\{u>0\}) + \int_{\Omega\cap \{ u>0\}} |\nabla u|^2,$$ where
$P_\Omega(\{u>0\})$ is the perimeter of the set $\{ u>0\}$ relative to the domain $\Omega$, while
in the case of one-phase solutions ($V\equiv 0$) of problem (\ref{eins})-(\ref{vier}), where we extend $\phi$ by the
value $0$ to the fluid phase and we consider $B=0$, the energy takes the form
\begin{equation}
- P_\Omega(\{\phi >0\}) + \int_{\Omega\cap \{ \phi>0\}} |\nabla \phi|^2.\label{energy}\end{equation}
As a consequence, constructing solutions by minimising the energy makes perfect sense for
(\ref{acseq1})-(\ref{acseq2}) and leads in dimension $n\le 7$ to regular solutions
(see \cite{silvestre}), while there obviously exist no minimisers of energy (\ref{energy}).
Moreover, critical points of the energy (\ref{energy}) may have singularities even in dimension $2$,
an example of a singularity being the real part of the complex root, multiplied by a suitable
positive constant. In that example two components of the fluid phase touch and create a multiplicity $2$ interphase, so while the curvature and $[|\nabla \phi|^2]$ (which denotes the jump of $|\nabla \phi|^2$) are both zero outside the origin, $|\nabla \phi|^2$ is not zero. It is not difficult to find more evidence underlining the drastic difference in
the qualitative behaviour of solutions of (\ref{acseq1})-(\ref{acseq2}) and
(\ref{eins})-(\ref{vier}).

Another problem related to the system (\ref{eins})-(\ref{vier}) is, in two dimensions and under certain assumptions,
that satisfied by critical points of the Mumford-Shah functional (see
for example \cite{MR1857292} and \cite{MR2143526}): the Mumford-Shah equations
are (up to terms of lower order)
\begin{align*}
&\Delta m = 0 \textrm{ in }\Omega\setminus S_m,\\
&[|\nabla m|^2]= \kappa\textrm{ on the free discontinuity set }\Omega \cap S_m,\\
& \langle \nabla m, \nu\rangle = 0 \textrm{ on the free discontinuity set }\Omega \cap S_m,
\end{align*}
where $[|\nabla m|^2]$ denotes the jump of $|\nabla m|^2$. The sign of the jump
and that of the mean curvature are chosen such that
$m=0$ in a component $D$ of $\Omega$ implies that
$$|\nabla m|^2= -\kappa\textrm{ on }\Omega\cap \partial (\Omega\setminus D).$$
Note that the homogeneous Dirichlet condition (\ref{vier}) is replaced in the Mumford-Shah problem by a homogeneous
Neumann condition. That means that results available for solutions of the Mumford-Shah problem
can be applied to one-phase solutions $V$ of (\ref{eins})-(\ref{vier}), in which case
$\phi\equiv 0$. Another possibility, again in two dimensions, is to consider under certain assumptions
the harmonic conjugate of $\phi$ which, combined with $V$, yields a solution of
the Mumford-Shah equations. This also means that in two dimensions the results of the paper
\cite{WZ} are directly applicable to (\ref{eins})-(\ref{vier}).

Returning now to the system (\ref{eins})-(\ref{vier}), the following two examples in a $3$D axisymmetric setting will be useful in understanding scaling and magnitude of the functions involved:

1. Suppose first that $\phi$ and $V$ are homogeneous functions, that the fluid phase is a connected cone, and that the air phase is connected as well. Then the mean curvature scales like $1/r$, where $r$ is the distance to the origin. Condition (\ref{bern}) implies then that $\phi$ and $V$ are both homogeneous functions of order $1/2$,
and that the constant $B$ equals $0$. From 3.1 (3), it follows that there exists (up to rotation) a unique
solution $(\phi,0)$ with these properties, namely the well-known Taylor cone solution, for which the opening angle of the fluid cone is roughly $98.6^\circ$.

2. Secondly, suppose that the fluid phase is an infinite cylinder of radius $r$,
and that $\phi=1$ on the fixed outer boundary, which we assume to be the cylinder of radius $1$ centered around the
same axis as the fluid cylinder. We obtain that $\phi$ (the capacity potential of the cylinder)
equals $d \log(s/r)$, where $d=1/\log(1/r)$ and the variable $s$ is the distance to the common cylinder axis.
Condition (\ref{bern}) implies in this case that
$$\frac{1}{r^2 (\log(1/r))^2} - |\nabla V|^2 = \frac{1}{r}+B.$$
This suggests that the effect of surface tension is negligible as $r\to 0$, and, moreover, that the same holds
for thin fluid jets and cusps. We expect the shape of the fluid region to be determined by a balance of
electric and fluid flow on the free surface. (Note that in both examples gravity terms would be negligible as well.)

The second example is of particular interest to us, since both experiments and numerical simulations
show the formation of a stable fluid jet ending in a cusp (and then separating into a spray of droplets by electric repulsion).
Based on the second example, we will in the present paper neglect surface tension and study the
still nonlinear as well as nontrivial problem arising from (\ref{eins})-(\ref{vier}) by setting surface tension to zero. This is justified as we are in this paper interested in a regime close to a fluid cusp
and {\em ignore the separation into droplets}.

\begin{minipage}{\textwidth}
\begin{center}
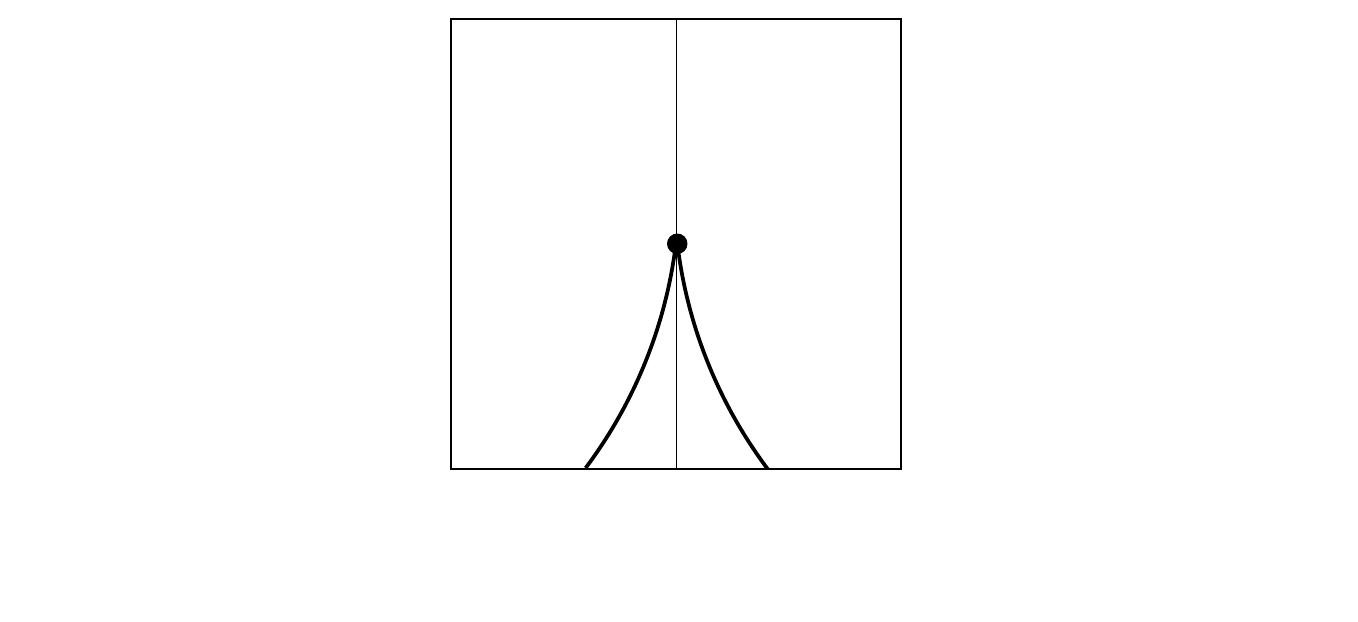
\end{center}
\captionof{figure}{Dynamics suggested by our analysis}\label{fig9}
\end{minipage}

So from now on, we will focus on the following set of equations:
\begin{align}
&\Delta \phi = 0 \textrm{ in the gas region,}\label{einsa}\\
&\Delta V =0 \textrm{ in the fluid region,}\\
&|\nabla \phi|^2 - |\nabla V|^2 = 0\textrm{ on the free surface of the fluid,}\label{berna}\\
& \phi=0\textrm{ on the free surface of the fluid,}\\
& \langle \nabla V,\nu\rangle = 0 \textrm{ on the free surface of the fluid.}\label{viera}
\end{align}

Moreover, we will restrict our analysis to an axisymmetric setting:
using cylindrical coordinates and combining the electric potential
and the Stokes stream function (see for example \cite[Exercise 4.18 (ii)]{fraenkel}), assumed to not change sign in the respective regions,
into a single function $u$, we obtain the free boundary problem
\begin{align} \label{intro_sol}
&\div \left(x_1 \nabla u(x_1,x_2)\right)=0 \textrm{ in the gas phase } \{ u>0\}, \non \\
&\div \left(\frac{1}{x_1} \nabla u(x_1,x_2)\right)=0 \textrm{ in the water phase } \{ u<0\}, \\
&x_1 {\vert \nabla u^+(x_1,x_2)\vert}^2 - \frac{1}{x_1}{\vert \nabla u^-(x_1,x_2)\vert}^2
=  0 \text{ on the free surface } \partial
\{ u >0\},\non\end{align}
where $u$ is defined in $\Om$, a subset of the right half-plane $\{(x_1,x_2) \ : \ x_1 \ge 0\}$ that is connected
and relatively open (in the sense that $\Om= \tilde{\Om}\cap \{(x_1,x_2) \ : \ x_1\ge 0\}$, where $\tilde{\Om}$ is an open set in $\R^2$), and the original velocity field is
$$ \left(-\frac{1}{x_1} \partial_2 u^- \cos \vartheta,
-\frac{1}{x_1} \partial_2 u^- \sin \vartheta ,
\frac{1}{x_1} \partial_1 u^-\right),$$
where $(X,Y,Z)=(x_1\cos\vartheta, x_1\sin\vartheta, x_2)$.

Note that although this problem may be viewed as a free boundary problem
with jumping diffusion coefficient $a(x,u(x))$, the problem {\em cannot} be transformed
by the Kirchhoff transform into a regular problem, since the
diffusion coefficients in the free boundary condition and the diffusion coefficients
in each phase are not compatible in the way required for that method to work.

Another approach that does not seem to work, is the derivation of a {\em frequency formula}.
Frequency formulas have been introduced to partial differential equations in the celebrated
results by F. Almgren \cite{almgren} and N. Garofalo-F. Lin \cite{GL}
and have been successfully applied to free boundary problems (see \cite{vwacta},
\cite{Varvaruca2012}, \cite{vwcpam}, \cite{GP} and \cite{GSVG}).
The fact that they can still be applied in the presence of a reaction term
(see \cite{vwacta},
\cite{Varvaruca2012} and \cite{vwcpam}) and that the problems studied in
\cite{vwacta},
\cite{Varvaruca2012} and \cite{vwcpam} are rather similar to the fluid-phase 
part of problem \eqref{intro_sol} gave high hope to the possibility of a frequency formula
for \eqref{intro_sol}. Unfortunately various candidates for frequency functions considered by the authors
turned out not to work. This is not a complete surprise as there are hitherto no known frequency
formulas for two-phase Stefan problems, and possibly the elliptic system \eqref{intro_sol} is more akin to
that group of problems. Still the methods we use in the present paper are somewhat related to
frequency formulas.

In our main result, we show that the energy satisfies a non-degeneracy condition and if the free boundary is given by an injective curve, then it must be asymptotically cusp-shaped.

\begin{theorem} Let $\gamma\in (0,1/2)$, and let $u$ be a weak solution of \eqref{intro_sol} in the sense of Definition \ref{weaksol}. Suppose that $0\in \D \{u>0\}$ and that $\D \{u>0\}\cap B_1^+$ is, in a neighborhood of $0$, a continuous injective curve $\sigma: [0,1) \rightarrow \R^2$, where $\sigma=(\sigma_1,\sigma_2)$ and $\sigma(0)=0$. Suppose, additionally, that there exist $r_0, C>0$ such that for each $r\in (0,r_0)$, 
\[
Cr^{2\gamma+1}\le \int_{B_r^+}\left(x_1|\nabla u^+|^2+\frac{1}{x_1}|\nabla u^-|^2\right)dx.
\]
Then,
\[
\lim\limits_{t\rightarrow 0+}\left|\frac{\sigma_1(t)}{\sigma_2(t)}\right|=0,
\]
that is, the free boundary is asymptotically cusp-shaped.
\end{theorem}

For general values of $\gamma$ we expect, in view of the results in \cite{vwcpam}, water-filled cusps to exist. However, to exclude water-filled cusps for small $\gamma$ remains a problem we leave to subsequent study.

The proof relies on a monotonicity formula of Alt-Caffarelli-Friedman type (section \ref{S:ACF}) as well as a polynomial expansion result by Caffarelli and Friedman for superlinear equations \cite{cf2}. Moreover, we use a Weiss-type monotonicity formula (section
\ref{S:weiss}, which we prove for any dimension $n$) in order to avoid infinite order vanishing in the proof of our main theorem.
Note that (although possible when assuming strict growth bounds from above and below) we do not use the monotonicity formula \ref{T:mono} to derive asymptotic homogeneity of the solution.

As a by-product of our proofs, we obtain a unique homogeneous solution of \eqref{einsa}-\eqref{viera} (see Remark \ref{C:beta}). For that solution the air phase is a convex cone. It seems to correspond to the well-known "Garabedian bubble" in fluid flow without electric field.

{\bfseries Acknowledgements}
We thank an anonymous referee for providing valuable feedback on a preliminary version of the paper.

\section{Notation}
\label{notation}
We use coordinates $(X,Y,Z)$ in the physical space
$\R^3$ 
as well as two-dimensional coordinates
$x=(x_1,x_2)$ together with partial derivatives $\partial_1,\partial_2$.
Cylindrical coordinates, as used in section \ref{S:intro}, are denoted with
$(X,Y,Z)=(x_1 \cos \vartheta, x_1 \sin \vartheta, x_2)$. We denote by $\langle x, y\rangle$ the Euclidean inner
product in $\R^n \times \R^n$, by $\vert x \vert$ the Euclidean norm
in $\R^n$, by $B_r(x^0):= \{x \in \R^n : \vert x-x^0 \vert < r\}$
the ball of center $x^0$ and radius $r$,
by $B^+_r(x^0):= \{x \in \R^n : x_1>0 \textrm{ and }\vert x-x^0 \vert < r\}$,
by $\partial B^+_r(x^0):= \{x \in \R^n : x_1>0 \textrm{ and }\vert x-x^0 \vert =r\}$
and $\R^n_+ := \{ (x_1,\dots,x_n) : x_1>0\}$ the positive parts.
\emph{Note that $\partial B^+_r(x^0)$ is not the topological boundary
of $B^+_r(x^0)$ and that $B^+_r(x^0)$ is not necessarily a half ball.}

We will use the notation
$B_r$ for $B_r(0)$ as well as $B^+_r$ for $B^+_r(0)$, and denote by $\omega_n$ the $n$-dimensional
volume of $B_1$. 
We will use the weighted $L^2$ and Sobolev spaces, for $U\subset \R^2_+$ open:

\[
L^2_w(U):=\Big\{ v\in L^2(U) \ : \ \int_{U}\left(x_1(v^+)^2+ \frac{1}{x_1}(v^-)^2\right)dx <\infty\Big\},
\]
\begin{align*}
V^{1,p}(U):=\Big\{ v & \in W^{1,p}(U) \ : \ \int_{U} \left(x_1|\nabla v^+|^p+\frac{1}{x_1}|\nabla v^-|^p\right)dx<\infty,\\ &\int_{U}\left(x_1(v^+)^p+\frac{1}{x_1}(v^-)^p\right)dx<\infty \Big\},
\end{align*}
as well as its local and surface versions, for $\Gamma\subseteq \D B_r^+$:
\begin{align*}
V^{1,p}_{\text{loc}}(\R^2_+):=\Big\{ v & \in W^{1,p}_{\text{loc}}(\R^2_+) \ : \ \int_{B_R^+} \left(x_1|\nabla v^+|^p+\frac{1}{x_1}|\nabla v^-|^p\right)dx<\infty,\\ &\int_{B_R^+}\left(x_1(v^+)^p+\frac{1}{x_1}(v^-)^p\right)dx<\infty, \ \text{ for every } R>0 \Big\},
\end{align*}
\[
L^2_w(\Gamma):=\Big\{ v\in L^2(\Gamma) \ : \ \int_{\Gamma}\left(x_1(v^+)^2+ \frac{1}{x_1}(v^-)^2\right)dS<\infty\Big\},
\]

\begin{align*}
V^{1,p}(\Gamma):=\Big\{ v & \in W^{1,p}(\Gamma) \ : \ \int_{\Gamma} \left(x_1|\nabla_{\tau} v^+|^p+\frac{1}{x_1}|\nabla_{\tau} v^-|^p\right)dS<\infty,\\
&\int_{\Gamma}\left(x_1(v^+)^p+\frac{1}{x_1}(v^-)^p\right)dS<\infty\Big\},
\end{align*}
where $dS$ denotes surface measure, and $\nabla_{\tau}v$ denotes the tangential gradient of $v$.
We will also denote by $V^{1,p}_0(\Gamma):=\{ v\in V^{1,p}(\Gamma) \ : \ v=0 \text{ on } \D \Gamma\}$.

We denote by $\chi_A$ the characteristic function
of a set $A$. For any real number $a$, the notation $a^+$
stands for $\max(a, 0)$ and $a^-$
stands for $\min(a, 0)$.  Also, ${\mathcal H}^s$ shall denote the
$s$-dimensional Hausdorff measure and by $\nu$ we will always refer to
the outer normal on a given surface.

\section{Notion of solution and eigenvalue considerations}\label{notion}
Assume $n=2$.
\begin{definition} We define $u\in V^{1,2}(\Omega)$ to be a {\itshape variational solution} of \eqref{intro_sol} if $u\in C^0(\Omega\cap \{x_1>0\})\cap C^2(\Omega\cap \{u\ne 0\}\cap\{x_1>0\})$,
\[
\lim\limits_{\substack{x\rightarrow x^0, \\ x\in \Om\cap\{u>0\}}}x_1u\D_1u =0 \  \ \ \ \ \text{and} \ \ \ \ \
\lim\limits_{\substack{x\rightarrow x^0, \\ x\in \Om\cap\{u<0\}}}u\frac{\D_1u}{x_1} =0
\]
for any $x^0\in \Om\cap \{x_1=0\}$, and the first variation with respect to domain variations of the functional
\[
E(v) := \int_{\Omega} \left( x_1 |\nabla v^+|^2+\frac{1}{x_1}|\nabla v^-|^2\right) dx
\]
vanishes at $v=u$, i.e.
\begin{align*}
0 & =-\frac{d}{d\varepsilon} E(u(x+\varepsilon \phi(x)))\Big|_{\varepsilon=0}\\
&= \int_{\Omega}\Big[\left( x_1|\nabla u^+|^2+\frac{1}{x_1}|\nabla u^-|^2\right)\text{div}\phi - 2x_1\nabla u^+D\phi\nabla u^+
-\frac{2}{x_1}\nabla u^- D\phi\nabla u^-\\
&+\left( |\nabla u^+|^2 - \frac{1}{x_1^2}|\nabla u^-|^2   \right)\phi_1\Big]dx
\end{align*}
for any $\phi=(\phi_1,\phi_2)\in C^1_0(\Omega,\R^2)$ such that $\phi_1=0$ on $\{x_1=0\}$.
\end{definition}

\subsection{Eigenvalue considerations} \ \ \ \ \ \ \

Assume $n=2$.

Given a domain $\Gamma \subseteq \D B_1^+$, we define
\begin{equation}\label{lambda}
\lambda^+(\Gamma) :=\inf_{0\le v\in V^{1,2}_0(\Gamma)}\frac{\int_{\Gamma}x_1|\nabla_{\tau} v|^2dS}{\int_{\Gamma}x_1v^2dS}, \ \ \ \text{ and } \ \ \ \lambda^-(\Gamma) :=\inf_{0\ge v\in V^{1,2}_0(\Gamma)}\frac{\int_{\Gamma}\frac{|\nabla_{\tau}v|^2}{x_1}dS}{\int_{\Gamma}\frac{v^2}{x_1}dS}.
\end{equation}

With the goal of obtaining information about $\lambda^{\pm}(\Gamma)$, which will be instrumental in Section \ref{S:ACF}, we start by making a series of remarks.

\begin{enumerate}

\item If $\Gamma_1\subseteq \Gamma_2\subseteq \D B_1^+$ are two open sets, then
\begin{equation}\label{complementary}
\lambda^{\pm}(\Gamma_1)\ge \lambda^{\pm}(\Gamma_2),
\end{equation}
since any $v\in V^{1,2}_0(\Gamma_1)$ can be extended as zero to $\Gamma_2$. Moreover, if $\Gamma_1 \neq\Gamma_2$, then the inequality in \eqref{complementary} is also strict.
\item Let $\Gamma_{\theta}$ be the arc starting at $0$ and ending at the angle $\theta$. We define the functions
$I^{\pm}(\theta):=\lambda^{\pm}(\Gamma_{\theta})$, which are continuous.

Note that the function which achieves the infimum in the definition of $\lambda^+(\Gamma_{\theta})$ may be extended to a homogeneous solution of $\text{div}(x_1\nabla \cdot )=0$ of degree $\alpha_+(\theta):=\frac{-1+\sqrt{1+4I^+(\theta)}}{2}$ (since the equation leads to $I^+(\theta)=\alpha_+(\theta)(\alpha_+(\theta)+1)$).
Similarly, the function achieving the infimum regarding $\lambda^-(\Gamma_{\theta})$ can be extended to a homogeneous solution of $\text{div}\left(\frac{1}{x_1}\nabla \cdot\right)=0$ of degree $\alpha_-(\theta):=\frac{1+\sqrt{1+4I^-(\theta)}}{2}$ (since the equation leads to $I^-(\theta)=\alpha_-(\theta)(\alpha_-(\theta)-1)$).

Combining \eqref{lambda} with \eqref{complementary} (which proves that $I^+(\pi-\cdot)$ is strictly increasing and that
$I^-$ is strictly decreasing), and the fact that
$I^+(\pi-)=0, I^-(\pi-)=2$ (the homogeneity is $2$ by explicit computation of solutions in the half-plane in \cite{vwcpam}) and that
$I^+(0+)=I^-(0+)=+\infty$, we conclude that there exists a unique $\theta_1\in (0,\pi)$ such that
\begin{equation}\label{A}
I^+(\pi-\theta_1)=I^-(\theta_1).
\end{equation}
Then \eqref{A} implies that there exists a unique $\theta_1\in (0,\pi)$ such that
\[
1+\alpha_+(\pi-\theta_1)=\alpha_-(\theta_1).
\]

We define
\begin{equation}\label{alpha*}
\alpha^*:=\alpha_+(\pi-\theta_1)= \frac{-1+\sqrt{1+4I^+(\pi-\theta_1)}}{2}.
\end{equation}
We call $\alpha^*$ \emph{matched homogeneity}.

\item Given a connected arc $\Gamma\subseteq \D B_1^+$, we make the change of coordinates $x_1=\sin\theta, x_2=\cos \theta$, so that if $f(\theta)=u(\sin \theta,\cos \theta)$, then for some $\theta_1<\theta_2\in (0,\pi)$,
\[
\int_{\Gamma}x_1u^2dS=\int_{\theta_1}^{\theta_2}\sin\theta f^2(\theta)d\theta,  \text{ and }
\int_{\Gamma}x_1|\nabla_{\tau}u|^2dS=\int_{\theta_1}^{\theta_2}\sin\theta(f'(\theta))^2d\theta,
\]
and similarly
\[
\int_{\Gamma}\frac{1}{x_1}u^2dS=\int_{\theta_1}^{\theta_2}\frac{1}{\sin\theta}f^2(\theta)d\theta,  \text{ and } \  \int_{\Gamma}\frac{1}{x_1}|\nabla_{\tau}u|^2dS=\int_{\theta_1}^{\theta_2}\frac{1}{\sin \theta}(f'(\theta))^2d\theta.
\]
Consider now $f$ vanishing at $\theta_1$ and $\theta_2$, and
let $g(\theta):=\sin\theta f(\theta)$.
Using $g'(\theta)=\cos \theta f(\theta)+\sin\theta f'(\theta)$, we compute
\begin{align*}
\int_{\theta_1}^{\theta_2}\frac{1}{\sin \theta}(g'(\theta))^2d\theta
&=\int_{\theta_1}^{\theta_2} \frac{1}{\sin \theta}\Big[\cos^2\theta f^2(\theta)+2\sin \theta\cos \theta f(\theta)f'(\theta)\\
&+\sin^2 \theta (f'(\theta))^2\Big]d\theta
\\
&=\int_{\theta_1}^{\theta_2}\frac{\cos^2\theta}{\sin \theta}f^2(\theta)d\theta+\int_{\theta_1}^{\theta_2}\sin \theta f^2(\theta)d\theta+\int_{\theta_1}^{\theta_2}\sin\theta(f'(\theta))^2d\theta
\\
&=\int_{\theta_1}^{\theta_2}\frac{1}{\sin \theta}f^2(\theta)d\theta+\int_{\theta_1}^{\theta_2}\sin \theta (f'(\theta))^2d\theta.
\end{align*}

This implies that if $U=\bigcup_{j=1}^\infty (\theta_1^j,\theta_2^j)$ and $f$ vanishes on $\theta_i^j$, for $i \in \{1,2\}, j\in \N$,
\begin{align}\label{lambda0}
&\frac{\int_U\frac{1}{\sin \theta}(g')^2d\theta}{\int_U\frac{1}{\sin \theta}g^2d\theta}
= \frac{\int_U\frac{1}{\sin \theta}f^2d\theta+\int_U\sin \theta (f')^2d\theta}{\int_U \sin \theta f^2d\theta}\\
&\ge 1 + \frac{\int_U\sin \theta (f')^2d\theta}{\int_U \sin \theta f^2d\theta}.
\end{align}
Note that as any open subset of $(0,\pi)$ can be written as a countable union of open
intervals, this property extends to each open $U\subseteq(0,\pi)$ which can be written as a countable union of open intervals such that $f$ vanishes at its endpoints.
This property also extends from the $g$ above to any $g$ such that $\int_{\theta_1}^{\theta_2}\frac{1}{\sin \theta}(g'(\theta))^2d\theta$ is finite. We obtain that
$\lambda^-(\Gamma)\ge 1+\lambda^+(\Gamma)$ for any open set $\Gamma \subseteq \D B_1^+$.

\item
Last, we use \cite{cs} (see the discussion following ~(12.8))
to deduce from the above that
\begin{align}\label{lambda2}
&\sqrt{\lambda^+(\Gamma)}+\sqrt{\lambda^-(\D B_1^+\setminus \Gamma)}\nonumber\\
&\ge \sqrt{\lambda^+(\Gamma)}+\sqrt{1+\lambda^+(\D B_1^+ \setminus \Gamma)}\nonumber\\
&\ge \sqrt{\lambda^+(\Gamma)}+\sqrt{\lambda^+(\D B_1^+ \setminus \Gamma)}
\ge 2.
\end{align}

We do not know whether the infimum of $\sqrt{\lambda^+(\Gamma)}+\sqrt{\lambda^-(\D B_1^+\setminus \Gamma)}$ is attained at a connected arc $\Gamma$ or not.
The available rearrangement techniques seem not to be applicable. In theory, it is possible that the infimum is attained at $\Gamma$ split into two disconnected components touching $0$ or $\pi$, respectively. We conjecture, however, that the infimum {\em is} attained at a connected arc.
\end{enumerate}

\section{Alt-Caffarelli-Friendman type monotonicity formula}\label{S:ACF}

In this section we prove an Alt-Caffarelli-Friedman type monotonicity formula appropriately adapted to our framework (see \cite{ACF} and \cite{cs}). We assume that $n=2$. We remark that the computations do not depend on any boundary conditions being satisfied on $\partial\{u>0\}$, and thus the result is applicable even in the case of curvature, which will be explored in a forthcoming paper.

Let $u\in V^{1,2}(\Omega)\cap C(\Omega\cap\{x_1>0\})\cap C^2(\Omega\cap\{u\neq 0\}\cap\{x_1>0\})$ satisfy
\begin{equation}\label{u}
\begin{cases}
\text{div}(x_1\nabla u)=0 \text{ in } \{u>0\},\\
\text{div}\Big(\frac{1}{x_1}\nabla u\Big) = 0 \text{ in } \{ u<0\}.\\
\end{cases}
\end{equation}

Let us define
\begin{equation}\label{beta*}
\beta^*:=\inf_{\substack{\Gamma \text{ open },\\ \Gamma \subseteq \D B_1^+}}\left((\lambda^+(\Gamma))^{1/2}+(\lambda^-(\D B_1^+\setminus \Gamma))^{1/2}\right).
\end{equation}

\begin{lemma}\label{L:beta*} $\beta^*\ge 2$.
\end{lemma}

\begin{proof} Follows from (\ref{lambda2}).
\end{proof}

\begin{theorem}\label{T:ACF}
Let $u \in V^{1,2}(\Omega)\cap C(\Omega\cap\{x_1>0\})\cap C^2(\Omega\cap\{u\neq 0\}\cap\{x_1>0\})$ solve \eqref{u} and be such that
\[
\lim\limits_{\substack{x\rightarrow x^0, \\ x\in \Om\cap\{u>0\}}}x_1u\D_1u =0 \  \ \ \ \ \text{and} \ \ \ \ \
\lim\limits_{\substack{x\rightarrow x^0, \\ x\in \Om\cap\{u<0\}}}u\frac{\D_1u}{x_1} =0,
\]
for any $x^0\in \Om\cap \{x_1=0\}$. Suppose that $0\in \Omega$ and let $r_0>0$ be such that $B_{r_0}^+\subset \Omega$. Let $\Phi:(0,r_0)\to\R$ be given by
\begin{equation}\label{ACF2}
 \Phi(r)=\Phi(r,u^+,u^-):=\frac{1}{r^{2\beta^*}}\int_{B_r^+}x_1|\nabla u^+|^2dx\int_{B_r^+}\frac{|\nabla u^-|^2}{x_1}dx\quad\text{for all }r\in (0,r_0).
\end{equation}
Then $r\mapsto \Phi(r)$ is nondecreasing on $(0,r_0)$.
\end{theorem}

\begin{proof} Note that, since the functions
\[
r\mapsto \int_{B_r^+} x_1|\nabla u^+|^2 dS \quad\text{ and }\quad
r\mapsto \int_{B_r^+}\frac{|\nabla u^-|^2}{x_1}dS
\]
are absolutely continuous on $(0,r_0)$, we obtain that, for a.e. $r\in (0,r_0)$,
\begin{align*}
\Phi'(r)&=\frac{1}{r^{2\beta^*}}\int_{\D B_r^+} x_1|\nabla u^+|^2dS\int_{B_r^+}\frac{|\nabla u^-|^2}{x_1}dx\\
&+ \frac{1}{r^{2\beta^*}}\int_{B_r^+}x_1|\nabla u^+|^2dx \int_{\D B_r^+}\frac{|\nabla u^-|^2}{x_1}dS\\
& -\frac{2\beta^*}{r^{2\beta^*+1}}\int_{B_r^+}x_1|\nabla u^+|^2dx\int_{B_r^+}\frac{|\nabla u^-|^2}{x_1}dx.
\end{align*}
It follows that, for a.e. $r\in (0,r_0)$,
\begin{equation}\label{phi'}
\frac{r\Phi'(r)}{\Phi(r)}= \frac{r\int_{\D B_r^+}x_1|\nabla u^+|^2dS}{\int_{B_r^+}x_1|\nabla u^+|^2dx} +\frac{r\int_{\D B_r^+}\frac{|\nabla u^-|^2}{x_1}dS}{\int_{B_r^+}\frac{|\nabla u^-|^2}{x_1}dx} -2\beta^*,
\end{equation}
and it suffices to prove that this quantity is nonnegative.

We first write the above denominators in a different form. Since $\text{div}(x_1\nabla u)=0$ in $\{u>0\}$, if we let $\varepsilon\rightarrow 0$ in the following integration by parts formula
\[
\int_{B_r^+\cap \{x_1>\varepsilon\}}x_1\langle \nabla u,\nabla \max(u-\varepsilon,0)^{1+\varepsilon}\rangle dx =\int_{\D B_r^+\cap \{x_1>\varepsilon\}}x_1\max(u-\varepsilon,0)^{1+\varepsilon}\langle \nabla u,\nu\rangle dS,
\]
we obtain
\begin{align*}
\int_{B_r^+}x_1|\nabla u^+|^2dx=\int_{\D B_r^+}x_1u^+\langle \nabla u^+,\nu\rangle dS.
\end{align*}
Similarly, using the fact that  $\text{div}\left(\frac{1}{x_1}\nabla u\right)=0$ in $\{u<0\}$, we obtain
\begin{align*}
\int_{B_r^+}\frac{1}{x_1}|\nabla u^-|^2 dx =\int_{\D B_r^+}\frac{1}{x_1}u^-\langle \nabla u^-,\nu\rangle dS.
\end{align*}

We therefore obtain from (\ref{phi'}) that, for a.e. $r\in (0, r_0)$,
\[
\frac{r\Phi'(r)}{\Phi(r)}= \frac{r\int_{\D B_r^+}x_1|\nabla u^+|^2dS}{\int_{\D B_r^+}x_1u^+\langle \nabla u^+,\nu\rangle dS} +\frac{r\int_{\D B_r^+}\frac{|\nabla u^-|^2}{x_1}dS}{ \int_{\D B_r^+}\frac{1}{x_1}u^-\langle \nabla u^-,\nu\rangle dS} -2\beta^*,
\]

In what follows, let us fix an arbitrary point $r\in (0,r_0)$ of differentiability for $\Phi$, and define $w\in V^{1,2} (B^+_{r_0/r})$ by $w(x)= u(rx)$ for all $x\in B^+_{r_0/r}$. Then
\be \label{phi'n}
\frac{r\Phi'(r)}{\Phi(r)}= \frac{\int_{\D B_1^+}x_1|\nabla w^+|^2dS}{\int_{\D B_1^+}x_1w^+\langle \nabla w^+,\nu\rangle dS} +\frac{\int_{\D B_1^+}\frac{|\nabla w^-|^2}{x_1}dS}{ \int_{\D B_1^+}\frac{1}{x_1}w^-\langle \nabla w^-,\nu\rangle dS} -2\beta^*.
\ee

Let us denote, for a.e. (with respect to $dS$) $x\in \D B_1^+$, $\partial_r w^{\pm}(x):=\langle \nabla w^{\pm}(x),\nu(x)\rangle$. Then, the Cauchy-Schwarz Inequality yields
\begin{equation}\label{denum+}
\int_{\D B_1^+}x_1w^+\partial_r w^+dS\\
\le \left(\int_{\D B_1^+} x_1(w^+)^2dS\right)^{1/2}\left(\int_{\D B_1^+}x_1(\partial_r w^+)^2dS\right)^{1/2},
\end{equation}
and
\begin{equation}\label{denum-}
\int_{\D B_1^+}\frac{1}{x_1} w^- \partial_r w^- dS \le \left(\int_{\D B_1^+}\frac{(w^-)^2}{x_1}dS\right)^{1/2}\left(\int_{\D B_1^+}\frac{(\partial_r w^-)^2}{x_1}dS\right)^{1/2}.
\end{equation}
On the other hand, since $|\nabla w^\pm|^2=|\partial_r w^\pm|^2+|\nabla_{\tau}w^\pm|^2$, it follows that
\begin{equation}\label{num+}
\begin{aligned}
\int_{\D B_1^+}x_1|\nabla w^+|^2 dS&= \int_{\D B_1^+}\left(x_1|\partial_r w^+|^2+x_1|\nabla_{\tau}w^+|^2\right)dS\\
&\ge 2\left(\int_{\D B_1^+} x_1(\partial_r w^+)^2dS\right)^{1/2}\left(\int_{\D B_1^+}x_1|\nabla_{\tau} w^+|^2dS\right)^{1/2}
\end{aligned}
\end{equation}
and
\begin{equation}\label{num-}
\int_{\D B_1^+}\frac{|\nabla w^-|^2}{x_1}dS\ge 2\left(\int_{\D B_1^+}\frac{(\partial_r w^-)^2}{x_1}dS\right)^{1/2}\left(\int_{\D B_1^+}\frac{|\nabla_{\tau}w^-|^2}{x_1}dS\right)^{1/2}.
\end{equation}

Using now the estimates \eqref{num+}, \eqref{num-}, \eqref{denum+} and \eqref{denum-} in \eqref{phi'}, we obtain that
\begin{equation}\label{phi'1}
\begin{aligned}
\frac{r\Phi'(r)}{\Phi(r)}&\ge \frac{\int_{\D B_1^+}\left(x_1|\partial_r w^+|^2+x_1|\nabla_{\tau}w^+|^2\right)dS}{\left(\int_{\D B_1^+} x_1(w^+)^2dS\right)^{1/2}\left(\int_{\D B_1^+}x_1(\partial_r w^+)^2dS\right)^{1/2}}\\
& +\frac{\int_{\D B_1^+}\left(\frac{(\partial_r w^-)^2}{x_1}+\frac{|\nabla_{\tau}w^-|^2}{x_1}\right)dS}{\left(\int_{\D B_1^+}\frac{(w^-)^2}{x_1}dS\right)^{1/2}\left(\int_{\D B_1^+}\frac{(\partial_r w^-)^2}{x_1}dS\right)^{1/2}} -2\beta^*\\
&\ge \frac{2\left(\int_{\D B_1^+} x_1(\partial_r w^+)^2dS\right)^{1/2}\left(\int_{\D B_1^+}x_1|\nabla_{\tau} w^+|^2dS\right)^{1/2}}{\left(\int_{\D B_1^+} x_1(w^+)^2dS\right)^{1/2}\left(\int_{\D B_1^+}x_1(\partial_r w^+)^2dS\right)^{1/2}}\\
&+\frac{2\left(\int_{\D B_1^+}\frac{(\partial_r w^-)^2}{x_1}dS\right)^{1/2}\left(\int_{\D B_1^+}\frac{|\nabla_{\tau}w^-|^2}{x_1}dS\right)^{1/2}}{\left(\int_{\D B_1^+}\frac{(w^-)^2}{x_1}dS\right)^{1/2}\left(\int_{\D B_1^+}\frac{(\partial_r w^-)^2}{x_1}dS\right)^{1/2}} -2\beta^*\\
&=\frac{2\left(\int_{\D B_1^+}x_1|\nabla_{\tau} w^+|^2dS\right)^{1/2}}{\left(\int_{\D B_1^+}x_1(w^+)^2dS\right)^{1/2}}+\frac{2\left(\int_{\D B_1^+}\frac{|\nabla_{\tau}w^-|^2}{x_1}dS\right)^{1/2}}{\left(\int_{\D B_1^+}\frac{(w^-)^2}{x_1}dS\right)^{1/2}} -2\beta^*.
\end{aligned}
\end{equation}

Recalling the definitions in \eqref{lambda}, we have therefore obtained that
\begin{equation}\label{char}
\frac{r\Phi'(r)}{\Phi(r)}\ge 2\left(\sqrt{\lambda^+(\D B_1^+\cap\{w>0\})}+\sqrt{\lambda^-(\D B_1^+\cap\{w<0\})}-\beta^*\right),
\end{equation}
and the definition of $\beta^*$ implies that $\Phi'(r)\geq 0$. Since $r$ was an arbitrary differentiability point for $\Phi$ in $(0,r_0)$, the claimed monotonicity of $\Phi$ follows.
\end{proof}

\begin{corollary}\label{C:ACFzero} Given $0\le\alpha < \beta^*$, where $\beta^*$ is defined as in \eqref{beta*},
\[
\frac{1}{r^{2\alpha}}\int_{B_r^+}x_1|\nabla u^+|^2dx\int_{B_r^+}\frac{1}{x_1}|\nabla u^-|^2dx\rightarrow 0 \ \text{ as } r\rightarrow 0.
\]
By Lemma \ref{L:beta*}, this holds in particular for $0\le\alpha < 2$.
\end{corollary}

\begin{proof} This is an immediate consequence of Theorem \ref{T:ACF}.

\end{proof}

\begin{corollary}\label{C:blowups} Given $0< 2\alpha+1 <\beta^*$, let $0<r_m\rightarrow 0+$ be a sequence such that the blow-up sequence
\begin{equation}\label{um}
u_m(x):=\frac{u^+(r_mx)}{r_m^{\alpha}}+ \frac{u^-(r_mx)}{r_m^{\alpha+1}}
\end{equation}
converges weakly in $V^{1,2}_{\textnormal{loc}}(\R^n_+)$ to a blow up limit $u_0$. Then either $u_0^+\equiv 0$ or $u_0^-\equiv 0$.
\end{corollary}

\begin{proof} We have that
\begin{align*}
\Phi(s,u_m^+,u_m^-)&= \frac{1}{s^{2\beta^*}}\int_{B_s^+}x_1|\nabla u_m^+|^2dx\int_{B_s^+}\frac{1}{x_1}|\nabla u_m^-|^2dx\\
&=\frac{r_m^{-2(2\alpha+1)}}{s^{2\beta^*}}\int_{B_{sr_m}^+}x_1|\nabla u^+|^2dx\int_{B_{sr_m}^+}\frac{1}{x_1}|\nabla u^-|^2dx\\
&=r_m^{2(\beta^*-2\alpha-1)}\Phi(r_ms,u^+,u^-).
\end{align*}
Since $\Phi(s,u_0^+,u_0^-)\le \liminf_{m\rightarrow +\infty}\Phi(s,u^+_m,u^-_m)$, and the right-hand side of the above expression converges to $0$ as $r_m\rightarrow 0$ (since the limit $\lim\limits_{t\rightarrow 0+}\Phi(t,u^+,u^-)$ exists by Theorem \ref{T:ACF}), we conclude that $u_0$ is one-phase.
\end{proof}

\section{Degenerate points}\label{S:weiss}

In this section we prove a version of the monotonicity formula in \cite{weisscpde},
\cite{vwacta} and \cite{vwcpam} adapted to our framework. As the computations hold in any dimension, we present the next theorem in this setting.

\begin{theorem}[Monotonicity formula]\label{T:mono} Let $u$ be a variational solution of \eqref{intro_sol}, let $x^0\in \Omega$ be such that $x_1^0=0$ and let
$\delta:=\text{dist}(x^0,\partial\Omega)/2$. Let, for any $r\in (0,\delta)$
\begin{equation}\label{I}
I(r)=\int_{B_r^+(x^0)} \left( x_1|\nabla u^+|^2+\frac{1}{x_1}|\nabla u^-|^2\right)dx,
\end{equation}

\begin{equation}\label{I+}
I_+(r)=\int_{B_r^+(x^0)} x_1|\nabla u^+|^2dx,
\end{equation}

\begin{equation}\label{I-}
I_-(r)=\int_{B_r^+(x^0)} \frac{1}{x_1}|\nabla u^-|^2dx,
\end{equation}

\begin{equation}\label{J+}
J_+(r)=\int_{\partial B_r^+(x^0)}x_1(u^+)^2 dS,
\end{equation}

\begin{equation}\label{J-}
J_-(r)=\int_{\partial B_r^+(x^0)}\frac{1}{x_1}(u^-)^2 dS.
\end{equation}

Given $\beta>0$, we define
\begin{equation}\label{M}
M(r)=r^{-2\beta-n+1}I(r)-\beta r^{-2\beta-n}J_+(r)-(\beta+1) r^{-2\beta-n}J_-(r).
\end{equation}
Then, for a.e. $r\in (0,\delta)$,
\begin{equation}\label{M'}
\begin{aligned}
M'(r)= & 2r^{-2\beta-n+1}\int_{\D B_r^+(x^0)}\Big[x_1\left(\langle \nabla u^+,\nu\rangle-\frac{\beta}{r}u^+\right)^2\\
&+\frac{1}{x_1}\left(\langle \nabla u^-,\nu\rangle-\frac{\beta+1}{r}u^-\right)^2\Big]dS.
\end{aligned}
\end{equation}
\end{theorem}

We will split the proof of this result into several Lemmas where we compute the derivatives of $I(r)$ and $J(r)$ for arbitrary $x^0\in \Omega$. We start by rewriting $I(r)$.

\begin{lemma}\label{L:I} Let $u$ be a variational solution of \eqref{intro_sol}, let $x^0\in \Omega$ and let $\delta=\text{dist}(x^0,\partial\Omega)/2$. Then for $r\in(0,\delta)$,
\begin{equation}\label{partsI}
I(r)=\int_{\partial B_r^+(x^0)} \left( x_1u^+\langle \nabla u^+,\nu\rangle +\frac{1}{x_1}u^-\langle \nabla u^-,\nu\rangle\right)dS.
\end{equation}
\end{lemma}

\begin{proof} This integration by parts formula follows by letting $\varepsilon\rightarrow 0$ in
\begin{align*}
&\int_{B_r^+(x^0)\cap\{x_1>\varepsilon\}}\left[ x_1\langle \nabla u,\nabla \max(u-\varepsilon,0)^{1+\varepsilon}\rangle+\frac{1}{x_1}\langle \nabla u,\nabla \min(u+\varepsilon,0)^{1+\varepsilon}\rangle \right] dx\\
&= \int_{\partial B_r^+(x^0)\cap\{x_1>\varepsilon\}}\left[ x_1\max(u-\varepsilon,0)^{1+\varepsilon}\langle \nabla u,\nu\rangle+ \frac{1}{x_1}\min(u+\varepsilon,0)^{1+\varepsilon}\langle \nabla u,\nu\rangle \right] dS.
\end{align*}
\end{proof}

\begin{lemma}\label{L:J'} Let $u$ be a variational solution of \eqref{intro_sol} and let $x^0\in \Omega$. Let $\delta=\text{dist}(x^0,\partial\Omega)/2$ in the case $x^0_1=0$ and $\delta=\min\{\text{dist}(x^0,\partial\Omega)/2,x^0_1\}$ when $x^0_1>0$. Then, for a.e. $r\in (0,\delta)$,
\begin{equation}\label{J+'}
\begin{aligned}
(J_+)'(r)&=\frac{n-1}{r}\int_{\D B_r^+(x^0)}x_1(u^+)^2dS+\frac{1}{r}\int_{\D B_r^+(x^0)}(x-x^0)_1(u^+)^2dS\\
&+2\int_{\D B_r^+(x^0)}x_1u^+\langle \nabla u^+,\nu\rangle dS,
\end{aligned}
\end{equation}

\begin{align*}\label{J-'}
(J_-)'(r)&=\frac{n-1}{r}\int_{\D B_r^+(x^0)}\frac{1}{x_1}(u^-)^2dS+2\int_{\D B_r^+(x^0)}\frac{1}{x_1}u^-\langle\nabla u^-,\nu\rangle dS\\
&-\frac{1}{r}\int_{\D B_r^+(x^0)}\frac{(x-x^0)_1}{x_1^2}(u^-)^2dS.
\end{align*}
Therefore, if $J(r)=J_+(r)+J_-(r)$, then
\begin{equation}\label{J'}
J'(r)=(n-1)\frac{J(r)}{r}+2I(r)+ \frac{1}{r}\int_{\partial B_r^+(x^0)}(x-x^0)_1\left[ (u^+(x))^2 -\frac{1}{x_1^2}(u^-(x))^2\right]dS.
\end{equation}
\end{lemma}

\begin{proof} Assume $x^0_1>0$. We make the change of variables $\frac{x-x^0}{r}=y$ and write
\begin{align*}
& J_+(r)=\int_{\D B_1}(ry+x^0)_1 u^+(ry+x^0)^2r^{n-1} dS, \\
& J_-(r)=\int_{\D B_1}\frac{1}{(ry+x^0)_1}u^-(ry+x^0)^2r^{n-1} dS,
\end{align*}
then
\begin{align*}
(J_+)'(r)&=\int_{\D B_1}y_1u^+(ry+x^0)^2r^{n-1}dS\\
&+(n-1)r^{n-2}\int_{\D B_1}(ry+x^0)_1u^+(ry+x^0)^2dS\\
&+2\int_{\D B_1}(ry+x^0)_1u^+(ry+x^0)\langle\nabla u^+(ry+x^0),y\rangle r^{n-1}dS\\
&=\frac{1}{r}\int_{\D B_r(x^0)}(x-x^0)_1(u^+)^2dS+2\int_{\D B_r(x^0)}x_1u^+\langle\nabla u^+,\nu\rangle dS\\
& +\frac{n-1}{r}\int_{\D B_r(x^0)}x_1(u^+)^2dS,
\end{align*}

\begin{align*}
(J_-)'(r)&= -\int_{\D B_1}\frac{y_1}{(ry+x^0)_1^2}u^-(ry+x^0)^2r^{n-1}dS\\
&+\int_{\D B_1}\frac{n-1}{(ry+x^0)_1}u^-(ry+x^0)^2r^{n-2}dS\\
&+2\int_{\D B_1}\frac{1}{(ry+x^0)_1}u^-(ry+x^0)\langle \nabla u^-(ry+x^0),y\rangle r^{n-1} dS\\
&=-\int_{\D B_r(x^0)}\frac{(x-x^0)_1}{rx_1^2}(u^-)^2dS+2\int_{\D B_r(x^0)}\frac{1}{x_1}u^-\langle \nabla u^-,\nu\rangle dS\\
&+\frac{n-1}{r}\int_{\D B_r(x^0)}\frac{1}{x_1}(u^-)^2dS.
\end{align*}

An analogous, simpler computation follows when $x^0_1=0$ noticing that the integrals should be appropriately computed on $\partial B_r^+$ and $\partial B_1^+$. 

\end{proof}

\begin{lemma}\label{L:I'} Let $u$ be a variational solution of \eqref{intro_sol} and $x^0\in \Omega$. Let $\delta=\text{dist}(x^0,\partial\Omega)/2$ in the case $x^0_1=0$ and $\delta=\min\{\text{dist}(x^0,\partial\Omega)/2,x^0_1\}$ when $x^0_1>0$. Then, for a.e. $r\in (0,\delta)$,
\begin{align*}
I'(r)&=\frac{n-2}{r}I(r)+2\int_{\partial B_r^+(x^0)}\left(x_1\langle \nabla u^+,\nu\rangle^2+\frac{1}{x_1}\langle \nabla u^-,\nu\rangle ^2\right)dS\\
&+\frac{1}{r}\int_{B_r^+(x^0)}(x-x^0)_1\left(|\nabla u^+|^2-\frac{1}{x_1^2}|\nabla u^-|^2\right)dx.
\end{align*}
In particular, if $x^0_1=0$, then
\begin{equation}\label{I'}
I'(r)=2\int_{\partial B_r^+(x^0)}\left(x_1\langle \nabla u^+,\nu\rangle^2+\frac{1}{x_1}\langle \nabla u^-,\nu\rangle ^2\right)dS+\frac{n-1}{r}I_+(r)+\frac{n-3}{r}I_-(r).
\end{equation}
\end{lemma}

\begin{proof}
For small positive $\tau>0$ and $\eta_{\tau}(t):=\max\{0,\min\{1,(r-t)/\tau\}\}$, we take after approximation
\[
\phi_{\tau}(x):=\eta_{\tau}(|x-x^0|)(x-x^0)
\]
as a test function in the definition of a variational solution. We obtain

\begin{align*}
0&=\int_{\Omega}\left(x_1|\nabla u^+|^2+\frac{1}{x_1}|\nabla u^-|^2\right)\left(\eta_{\tau}'(|x-x^0|)|x-x^0|+n\eta_{\tau}(|x-x^0|)\right)dx\\
&-2\int_{\Omega}\left(x_1|\nabla u^+|^2+\frac{1}{x_1}|\nabla u^-|^2\right)\eta_{\tau}(|x-x^0|)dx\\
&-2\int_{\Omega}x_1\left\langle \nabla u^+,\frac{x-x^0}{|x-x^0|}\right\rangle^2\eta_{\tau}'(|x-x^0|)|x-x^0|dx\\
&-2\int_{\Omega}\frac{1}{x_1}\left\langle \nabla u^-,\frac{x-x^0}{|x-x^0|}\right\rangle^2\eta_{\tau}'(|x-x^0|)|x-x^0|dx\\
&+\int_{\Om}\left(|\nabla u^+|^2-\frac{1}{x^2}|\nabla u^-|^2\right)\eta_{\tau}(|x-x^0|)(x_1-x_1^0)dx.
\end{align*}

Letting $\tau\rightarrow 0$ we conclude that

\begin{align*}
0&=n\int_{B_r^+(x^0)}\left(x_1|\nabla u^+|^2+\frac{1}{x_1}|\nabla u^-|^2\right)dx
\\
&-r\int_{\partial B_r^+(x^0)}\left(x_1|\nabla u^+|^2+\frac{1}{x_1}|\nabla u^-|^2\right)dS\\
&-2\int_{B_r^+(x^0)}\left(x_1|\nabla u^+|^2+\frac{1}{x_1}|\nabla u^-|^2\right)dx
\\
&+2r\int_{\partial B_r^+(x^0)}\left(x_1\langle \nabla u^+,\nu\rangle^2+\frac{1}{x_1}\langle \nabla u^-,\nu\rangle^2\right)dS\\
&+\int_{B_r^+(x^0)}(x-x^0)_1\left(|\nabla u^+|^2-\frac{1}{x_1^2}|\nabla u^-|^2\right)dx.
\end{align*}

Hence

\begin{align*}
I'(r)&=\int_{\partial B_r^+(x^0)} \left(x_1|\nabla u^+|^2+\frac{1}{x_1}|\nabla u^-|^2\right)dS=\frac{n-2}{r}I(r)\\
&+2\int_{\partial B_r(x^0)}\left(x_1\langle \nabla u^+,\nu\rangle^2+\frac{1}{x_1}\langle \nabla u^-,\nu\rangle^2\right)dS\\
&+\frac{1}{r}\int_{B_r^+(x^0)}(x-x^0)_1\left(|\nabla u^+|^2-\frac{1}{x_1^2}|\nabla u^-|^2\right)dx.
\end{align*}
\end{proof}

\begin{proof}(of Theorem \ref{T:mono})
Recall that
\[
M(r)=r^{-2\beta-n+1}I(r)-\beta r^{-2\beta-n}J_+(r)-(\beta+1)r^{-2\beta-n}J_-(r).
\] By Lemmas \ref{L:J'} and \ref{L:I'}, we have that
\begin{align*}
M'&(r)  = (-2\beta-n+1) r^{-2\beta-n}I(r)+r^{-2\beta-n+1}I'(r)+\beta(2\beta+n)r^{-2\beta-n-1}J_+(r)\\
&-\beta r^{-2\beta-n}(J_+)'(r)+(\beta+1)(2\beta+n)r^{-2\beta-n-1}J_-(r)-(\beta+1) r^{-2\beta-n}(J_-)'(r)\\
&=(-2\beta-n+1) r^{-2\beta-n}I(r)\\
&+2r^{-2\beta-n+1}\int_{\partial B_r^+(x^0)}\left(x_1\langle \nabla u^+,\nu\rangle^2+\frac{1}{x_1}\langle \nabla u^-,\nu\rangle^2\right)dS\\
&+(n-1)r^{-2\beta-n}I_+(r)+(n-3)r^{-2\beta-n}I_-(r)+\beta(2\beta+n)r^{-2\beta-n-1}J_+(r)\\
&-\beta r^{-2\beta-n}\Big(2\int_{\D B_r^+(x^0)}x_1u^+\langle\nabla u^+,\nu\rangle dS+ \frac{n}{r}\int_{\partial B_r^+(x^0)}x_1(u^+)^2dS\Big)\\
&+(\beta+1)(2\beta+n)r^{-2\beta-n-1}J_-(r)-(\beta+1)r^{-2\beta-n}2\int_{\D B_r^+(x^0)}\frac{1}{x_1}u^-\langle\nabla u^-,\nu\rangle dS\\
&-(\beta+1)r^{-2\beta-n}+\frac{n-2}{r}\int_{\D B_r^+(x^0)}\frac{1}{x_1}(u^-)^2dS\\
&=-4\beta r^{-2\beta-n}\int_{\D B_r^+(x^0)}x_1u^+\langle \nabla u^+,\nu\rangle dS \\
&-4(\beta +1) r^{-2\beta-n}\int_{\D B_r^+(x^0)}\frac{1}{x_1}u^-\langle \nabla u^-,\nu\rangle dS \\
&+2r^{-2\beta-n+1} \int_{\partial B_r^+(x^0)}\left(x_1\langle \nabla u^+,\nu\rangle^2+\frac{1}{x_1}\langle \nabla u^-,\nu\rangle^2\right)dS\\
&+2\beta^2 r^{-2\beta-n-1}\int_{\D B_r^+(x^0)}x_1 (u^+)^2dS+2(\beta+1)^2r^{-2\beta-n-1}\int_{\D B_r^+(x^0)}\frac{1}{x_1}(u^-)^2dS\\
&=2r^{-2\beta-n+1}\int_{\D B_r^+(x^0)}\Big[x_1\left(\langle \nabla u^+,\nu\rangle-\frac{\beta}{r}u^+\right)^2\\
&+\frac{1}{x_1}\left(\langle \nabla u^-,\nu\rangle-\frac{\beta+1}{r}u^-\right)^2\Big]dS.
\end{align*}

\end{proof}

\section{Blow-up limits}

Let $0 \in \D \{u>0 \}$. Given $\beta \in (0,+\infty)$, assume the following growth conditions:
\begin{equation}\label{growth+}
|\nabla u^+| \le C|x|^{\beta-1} \ \ \text{in } B_{r_0}^+,
\end{equation}
\begin{equation}\label{growth-}
|\nabla u^-|\le C \sqrt{x_1}|x|^{\beta-1/2} \ \ \text{in } B_{r_0}^+,
\end{equation}
for some $r_0>0$ sufficiently small.

\begin{proposition}\label{P:homog}
Let $u$ be a variational solution of \eqref{intro_sol} and assume that $0\in\Omega\cap \partial\{u>0\}$. Furthermore, given $\beta\in (0,1/2)$, assume that \eqref{growth+} and \eqref{growth-} hold. Then
\begin{enumerate}
\item The limit $M(0+)=\lim\limits_{r\rightarrow 0+}M(r)$ exists and is finite.
\item Let $0<r_m\rightarrow 0+$ be such that the blow up sequence
\[
u_m(x):=\frac{u^+(r_mx)}{r_m^{\beta}}+\frac{u^-(r_mx)}{r_m^{\beta+1}}
\]
converges weakly in $V^{1,2}_{\textnormal{loc}}(\R^n_+)$ to a blow up limit $u_0$. Then $u_0^+$ is a homogeneous function of degree $\beta$ and $u_0^-$ is a homogeneous function of degree $\beta+1$ in $\R^n_+$, i.e., $u_0^+(\lambda x)=\lambda^{\beta}u_0^+(x)$ and $u_0^-(\lambda x)=\lambda^{\beta+1}u_0^-(x).$
\item Let $(u_m)_m$ be a converging sequence as in (2). Then $(u_m)_m$ converges strongly in $V^{1,2}_{\textnormal{loc}}(\R^n_+)$.
\item $M(0+)=0.$
\end{enumerate}
\end{proposition}

\begin{proof}
\begin{enumerate}
\item By Theorem \ref{T:mono}, $M(r)$ is a monotone increasing function. Equations \eqref{growth+} and \eqref{growth-} imply that $|u^+(x)|\le C|x|^{\beta}$ and $|u^-(x)|\le C\sqrt{x_1}|x|^{\beta+1/2}$. Since
\[
M(r)\ge -\beta r^{-2\beta-n}J_+(r)-(\beta+1)r^{-2\beta-n}J_-(r),
\]
we conclude that $M(r)$ is a bounded function and $M(0+)$ exists.
\item Let $0<\sigma,\delta<\infty$ and $x\in \{x_1>\delta\}$. If $\beta \in (0, 1)$, by \eqref{growth+}
\[
 |\nabla u^+(r_mx)|\le C |r_mx|^{\beta-1}\le C\delta^{\beta-1} r_m^{\beta-1},
 \]
 while if $\beta\ge 1$ we clearly have $|\nabla u^+(r_mx)|\le Cr_m^{\beta-1}$. Similarly, by \eqref{growth-}, if $\beta\in (0,1/2)$,
 \[
 |\nabla u^-(r_mx)|\le C\delta^{\beta-1/2} r_m^{\beta},
 \]
and if $\beta\ge 1/2$ we have $|\nabla u^-(r_mx)|\le Cr_m^{\beta}$. Therefore the sequence $(u_m)_m$ is bounded in $C^{0,1}(B_{\sigma}^+\cap \{x_1>\delta\})$, since \eqref{growth+} and \eqref{growth-} hold and
\[
\nabla u_m(x)=\frac{\nabla u^+(r_mx)}{r_m^{\beta-1}} +\frac{\nabla u^-(r_mx)}{r_m^{\beta}}.
\]
Consequently, up to a subsequence, $u_m$ converges locally uniformly to $u_0$. Given $0<\tau<\sigma<\infty$, we can write \eqref{M'} as

\begin{align*}
M(\sigma r_m)-&M(\tau r_m)=\int_{r_m\tau}^{r_m\sigma}M'(r)dr\\
&=\int_{r_m\tau}^{r_m\sigma}2r^{-2\beta-n+1}\int_{\D B_r^+}\Big[x_1\left(\langle \nabla u^+,\nu\rangle -\frac{\beta}{r}u^+\right)^2\\
&+\frac{1}{x_1}\left(\langle \nabla u^-,\nu\rangle -\frac{\beta+1}{r}u^-\right)^2\Big]dS dr\\
&=\int_{\tau}^{\sigma}2(r_ms)^{-2\beta-n+1}r_m\int_{\D B_{r_ms}^+}\Big[x_1\left(\langle \nabla u^+,\nu\rangle -\frac{\beta}{r_ms}u^+\right)^2\\
&+\frac{1}{x_1}\left(\langle \nabla u^-,\nu\rangle -\frac{\beta+1}{r_ms}u^-\right)^2\Big]dS ds\\
&= \int_{\tau}^{\sigma}2s^{-2\beta-n+1}r_m^{-2\beta+1}\int_{\D B_s^+}\Big[ r_my_1\left(\nabla u^+(r_my),\nu\rangle-\frac{\beta}{r_ms}u^+(r_my)\right)^2\\
&+\frac{1}{r_my_1}\left(\langle \nabla u^-(r_my),\nu\rangle -\frac{\beta+1}{r_ms}u^-(r_my)\right)^2 \Big]dS ds\\
&=\int_{\tau}^{\sigma}2s^{-2\beta-n+1}\int_{\D B_s^+}\Big[y_1\left(\langle \nabla u_m^+,\nu\rangle-\frac{\beta}{s}u_m^+\right)^2\\
&+\frac{1}{y_1}\left(\langle \nabla u_m^-,\nu\rangle -\frac{\beta+1}{s}u_m^-\right)^2\Big] dS ds\\
&=2\int_{B_{\sigma}^+\setminus B_{\tau}^+}|x|^{-2\beta-n-1}\Big[x_1\left(\langle\nabla u_m^+,x\rangle-\beta u_m^+\right)^2\\
&+\frac{1}{x_1}\left(\langle \nabla u_m^-,x\rangle -(\beta+1)u_m^-\right)^2\Big]dx.
\end{align*}
Letting $m\rightarrow\infty$ , since $M(0+)$ exists by (1), we conclude that
\begin{align*}
0&\ge\\
& 2\int_{B_{\sigma}^+\setminus B_{\tau}^+}|x|^{-2\beta-n-1}\left[x_1\left(\langle \nabla u_0^+,x\rangle -\beta u_0^+\right)^2+\frac{1}{x_1}\left(\langle \nabla u_0^-,x\rangle-(\beta+1)u_0^-\right)^2\right]dx,
\end{align*}
from where the homogeneity of $u_0$ follows.
\item
In view of the weak convergence, to prove the strong convergence of $u_m$ in $V^{1,2}_{\text{loc}}(\R^n_+)$ it suffices to prove that 
\[
\int_{\R^n_+}x_1|\nabla u_m^+|^2dx\rightarrow \int_{\R^n_+}x_1|\nabla u_0^+|^2dx
\] 
and 
\[
\int_{\R^n_+}\frac{1}{x_1}|\nabla u_m^-|^2dx\rightarrow \int_{\R^n_+}\frac{|\nabla u_0^-|^2}{x_1}dx.
\]
 Let $\delta:=\text{dist}(0,\D\Omega)$. We have that
\[
\text{div}\left(x_1\nabla u_m\right)=0 \ \ \ \text{ in } B_{\frac{\delta}{r_m}}^+\cap\{u_m> 0\}
\]
\[
\text{div}\left(\frac{1}{x_1}\nabla u_m\right)=0 \ \ \ \text{ in } B_{\frac{\delta}{r_m}}^+\cap \{u_m< 0\}
\]
Since $u_m$ converges locally uniformly to $u_0$ and $u_0$ is continuous, we conclude that in $\{ u_0 >0\}$ we have that
\begin{equation}\label{u0}
\text{div}(x_1\nabla u_0)=0
\end{equation}
and in $\{u_0< 0\}$ we have that
\begin{equation}\label{v0}
\text{div}\left(\frac{1}{x_1}\nabla u_0\right)=0.
\end{equation}
Let $\eta\in C_0^1(\R^n)$. An argument similar to that of the proof of \eqref{partsI} combined with \eqref{u0} and \eqref{v0} leads to
\begin{align*}
\int_{\R^n_+}x_1|\nabla u_m^+|^2\eta dx& \longrightarrow \int_{\R^n_+}x_1|\nabla u_0^+|^2\eta dx,
\end{align*}
and 
\[
\int_{\R^n_+}\frac{1}{x_1}|\nabla u_m^-|^2\eta dx \longrightarrow \int_{\R^n_+}\frac{1}{x_1}|\nabla u_0^-|^2\eta dx,
\]
from which the conclusion follows.
\item Let us take a sequence $r_m\rightarrow 0$ such that $(u_m)_m$ defined as in (2) converges weakly in $V^{1,2}_{\text{loc}}(\R^n_+)$ to $u_0$. Using the strong convergence proved in (3) and the homogeneity of $u_0^{\pm}$ proved in (2), we conclude that
\begin{align*}
&M(0+)=\lim\limits_{m\rightarrow\infty} M(r_m)\\
&=\lim\limits_{m\rightarrow\infty} \Big[r_m^{-2\beta-n+1}\int_{B_{r_m}^+}\left(x_1|\nabla u^+|^2+\frac{1}{x_1}|\nabla u^-|^2\right)dx\\
& -\beta r^{-2\beta-n}\int_{\D B_{r_m}^+}x_1(u^+)^2dS-(\beta+1)r_m^{-2\beta-n}\int_{\D B_{r_m}^+}\frac{1}{x_1}(u^-)^2dS\Big]\\
&=\lim\limits_{m\rightarrow \infty}\int_{B_1^+}\left(x_1|\nabla u_m^+|^2+\frac{1}{x_1}|\nabla u_m^-|^2\right)dx-\beta\int_{\D B_1^+}x_1(u_m^+)^2dS\\
&-(\beta+1)\int_{\D B_1^+}\frac{1}{x_1}(u_m^-)^2dS\\
&=\int_{B_1^+}\left(x_1|\nabla u_0^+|^2+\frac{1}{x_1}|\nabla u_0^-|^2\right)dx-\beta\int_{\D B_1^+}x_1 (u_0^+)^2dS\\
&-(\beta+1)\int_{\D B_1^+}\frac{1}{x_1}(u_0^-)^2dS=0.
\end{align*}
\end{enumerate}
\end{proof}

\begin{remark}\label{C:beta}
Let $\beta\in (0,+\infty)$. Assume that for some $r_m\rightarrow 0+$, $(u_m)_m$ (defined as in (2) of Proposition \ref{P:homog}) is bounded in $V^{1,2}_{\text{loc}}(\R^n_+)$. Moreover, assume there exists $c>0$ such that for every $m$ large
\[
c\le \int_{\D B_1^+}\left(x_1(u_m^+)^2+\frac{1}{x_1}(u_m^-)^2\right)dS.
\]
By Proposition \ref{P:homog}, $(u_m)_m$ converges in $V^{1,2}_{\text{loc}}(\R^n_+)$ to $u_0$, a two-phase solution of \eqref{intro_sol} such that $u_0^+$ is homogeneous of degree $\beta$ and $u_0^-$ of degree $\beta+1$.

Assume that $\{u_0>0\}\cap \D B_1^+$ and $\{u_0<0\}\cap \D B_1^+$ are connected sets. Restricted to the unit sphere
the functions $u_0^+$ and $u_0^-$ are eigenfunctions, therefore $\lambda^+(\D B_1^+\cap \{u_0>0\})=\beta(\beta+1)$ and $\lambda^-(\D B_1^+\cap \{u_0<0\})=\beta(\beta+1)$. Consequently, $\lambda^+(\D B_1^+\cap \{u_0>0\})=\lambda^-(\D B_1^+\cap \{u_0<0\})$, hence $\beta=\alpha^*$, where $\alpha^*$ was defined in \eqref{alpha*}.
\end{remark}

The following Lemma, which contains a growth estimate for the solution, will not be used in the sequel.

\begin{lemma}\label{L:JandM} Under the same assumptions of Proposition \ref{P:homog}, we have the follo\-wing relationship between $J(r)$ and $M(r)$:
\[
(r^{-2\beta-n}J(r))'=\frac{2}{r}M(r).
\]
In particular, using theorem \ref{T:mono} we conclude that
\begin{equation}\label{rJ}
(r^{-2\beta-n}J(r))'\ge \frac{2}{r}M(0+).
\end{equation}
Therefore, if $0<r_m\rightarrow 0+$ is such that the blow up sequence
\[
u_m(x):=\frac{u^+(r_mx)}{r_m^{\beta}}+\frac{u^-(r_mx)}{r_m^{\beta+1}}
\]
is bounded in $V^{1,2}_{\text{loc}}(\R^n_+)$, by Proposition \ref{P:homog}, $M(0+)=0$, hence $r^{-2\beta-n}J(r)$ is monotone non-decreasing.
\end{lemma}

\begin{proof}
\begin{align*}
(r^{-2\beta-n}J(r))'&=-(2\beta+n)r^{-2\beta-n-1}J(r)\\
&+r^{-2\beta-n}\left(\frac{n}{r}J_+(r)+2I(r)+\frac{n-2}{r}J_-(r)\right)\\
&=2r^{-2\beta-n}I(r)-2\beta r^{-2\beta-n-1}J_+(r)-2(\beta+1)r^{-2\beta-n-1}J_-(r)\\
&=\frac{2}{r}M(r).
\end{align*}
\end{proof}

\section{Convergence to cusps}

In this section we assume that $n=2$ and that the free boundary is, in a neighborhood of the origin, a continuous injective curve, and we study when it must be asymptotically cusp-shaped.

We start with a Lemma which will be used later on.

\begin{lemma}\label{L:cacc} Let $u$ be a variational solution of \eqref{intro_sol} in $B_2^+$. There exists $C>0$ such that
\[
\int_{B_1^+}x_1|\nabla u^+|^2dx \le C \int_{B_2^+\setminus B_1^+} x_1(u^+)^2dx,
\]
\[
 \int_{B_1^+}\frac{1}{x_1}|\nabla u^-|^2dx \le C\int_{B_2^+\setminus B_1^+}\frac{1}{x_1}(u^-)^2dx.
\]
\end{lemma}

\begin{proof} Let $\eta\in C^{\infty}_0(B_2)$ be such that $\eta\equiv 1$ in $B_1$, $0\le \eta\le 1$ and $|\nabla \eta|\le C$, for some dimensional constant $C>0$. Proceeding as in the proof of Lemma \ref{L:I} and using \eqref{intro_sol}, an integration by parts argument leads to
\begin{align*}
\int_{B_2^+}x_1|\nabla u^+|^2\eta^2dx &=-2\int_{B_2^+}x_1u^+\eta\langle \nabla u^+,\nabla \eta\rangle dx \\
&\le 2C\int_{B_2^+\setminus B_1^+}x_1 u^+\eta|\nabla u^+|dx\\
&\le C^2\int_{B_2^+\setminus B_1^+}x_1(u^+)^2dx+\int_{B_2^+\setminus B_1^+}x_1\eta^2|\nabla u^+|^2dx,
\end{align*}
hence
\[
\int_{B_1^+}x_1|\nabla u^+|^2dx \le C^2\int_{B_2^+\setminus B_1^+}x_1(u^+)^2.
\]
A similar argument holds for $u^-$.
\end{proof}

\begin{definition}\label{weaksol} We define $u \in V^{1,2}(\Omega)$ to be a \emph{weak solution} of \eqref{intro_sol} if the following are satisfied: $u$ is a variational solution of \eqref{intro_sol} and the topological free boundary $\partial\{u>0\}\cap\Omega\cap\{x_1>0\}$ is locally a $C^{2,\alpha}$ curve.
\end{definition}

Let $\gamma\in (0,1/2)$.

\begin{theorem}\label{T:main} Let $\gamma\in (0,1/2)$, and let $u$ be a weak solution of \eqref{intro_sol} in the sense of Definition \ref{weaksol}. Suppose that $0\in \D \{u>0\}$ and that $\D \{u>0\}\cap B_1^+$ is, in a neighborhood of $0$, a continuous injective curve $\sigma: [0,1) \rightarrow \R^2$, where $\sigma=(\sigma_1,\sigma_2)$ and $\sigma(0)=0$. Suppose, additionally, that there exist $r_0, C>0$ such that for each $r\in (0,r_0)$, 
\begin{equation}\label{bound}
Cr^{2\gamma+1}\le \int_{B_r^+}\left(x_1|\nabla u^+|^2+\frac{1}{x_1}|\nabla u^-|^2\right)dx.
\end{equation}
Then,
\[
\lim\limits_{t\rightarrow 0+}\left|\frac{\sigma_1(t)}{\sigma_2(t)}\right|=0,
\]
that is, the free boundary is asymptotically cusp-shaped.
\end{theorem}

\begin{proof}

If the result does not hold, then there exist $\varepsilon >0$ and sequences $(r_m)_m, (\xi_m)_m$ such that $r_m\rightarrow 0$, $\xi_m\in S^1$ and $r_m\xi_m\in \partial \{u>0\}$, but $(\xi_m)_1>\varepsilon|\xi_m|_2$ for all $m$.

We start by analyzing following sequence of rescalings of $u$ (blow-up sequence):
\begin{equation}\label{critical}
u_m(x):=\frac{u^+(r_mx)}{r_m^{\gamma}}+\frac{u^-(r_mx)}{r_m^{\gamma+1}}.
\end{equation}
Note that, by \eqref{bound}, up to a subsequence, either $\frac{C}{2}\le \int_{B_1^+}x_1|\nabla u_m^+|^2dx$, or $\frac{C}{2}\le \int_{B_1^+}\frac{1}{x_1}|\nabla u_m^-|^2dx$. For definiteness  we will assume that the first of these estimates holds, since the analysis is similar in the other case.

By Theorem \ref{T:ACF}, since $\beta^*>2\gamma+1$, we have that
\begin{align*}
\int_{B_1^+}&x_1|\nabla u_m^+|^2dx\int_{B_1^+}\frac{1}{x_1}|\nabla u_m^-|^2dx
=r_m^{-2(2\gamma+1)}\int_{B_{r_m}^+}x_1|\nabla u^+|^2dx\int_{B_{r_m}^+}\frac{1}{x_1}|\nabla u^-|^2dx\\
&=r_m^{2(\beta^*-2\gamma-1)}\Phi(r_m,u^+,u^-)\le r_m^{2(\beta^*-2\gamma-1)}\Phi(1,u^+,u^-)\rightarrow 0,
\end{align*}
and therefore
\begin{equation}\label{ACFconc}
\int_{B_1^+}x_1|\nabla u_m^+|^2dx\int_{B_1^+}\frac{1}{x_1}|\nabla u_m^-|^2dx\rightarrow 0.
\end{equation}
By \eqref{bound}, we may assume that $\int_{B_1^+}\frac{1}{x_1}|\nabla u_m^-|^2dx\rightarrow 0,$ which will be used below.

By Lemma \ref{L:cacc} there exists universal $C>0$ such that
\begin{equation}\label{caccioppoli}
\int_{B_1^+}x_1|\nabla u_m^+|^2dx\le C\int_{B_2^+\setminus B_1^+}x_1(u_m^+)^2dx.
\end{equation}
Combining \eqref{caccioppoli} and \eqref{bound} we conclude that there exist $c>0$ and $t_m\in (1,2)$ such that, for every $m$, we have
\begin{equation}\label{cacc2}
c\le \int_{\D B_{t_m}^+}x_1(u_m^+)^2dS.
\end{equation}

We now define the following sequence of rescalings:
\begin{align*}
v_m(x)&=\frac{u^+_m(x)+u_m^-(x)}{\sqrt{\int_{\D B_{t_m}^+}\left(x_1(u_m^+)^2+\frac{1}{x_1}(u_m^-)^2\right)dS}}.
\end{align*}
Notice that trivially $\int_{\D B_{t_m}^+}\left(x_1(v_m^+)^2+\frac{1}{x_1}(v_m^-)^2\right)dS=1$.

By Theorem \ref{T:mono} applied with $\beta=\gamma$,
\begin{align*}
M&(r_mt_m)=(r_mt_m)^{-2\gamma-1}\int_{B_{t_mr_m}^+}\left(x_1|\nabla u^+|^2+\frac{1}{x_1}|\nabla u^-|^2\right)dx\\
&-\gamma (r_mt_m)^{-2\gamma-2}\int_{\D B_{r_mt_m}^+}x_1(u^+)^2 dS\\
&-(\gamma+1)(r_mt_m)^{-2\gamma-2}\int_{\D B_{r_mt_m}^+}\frac{1}{x_1}(u^-)^2dS\\
&=t_m^{-2\gamma-1}\int_{\D B_{t_m}^+}\left(x_1(u_m^+)^2+\frac{1}{x_1}(u_m^-)^2\right)dS\int_{B_{t_m}^+}\left(x_1|\nabla v_m^+|^2+\frac{1}{x_1}|\nabla v_m^-|^2\right)dx\\
&-t_m^{-2\gamma-2}\gamma\int_{\D B_{t_m}^+}x_1(u_m^+)^2dS-t_m^{-2\gamma-2}(\gamma+1)\int_{\D B_{t_m}^+}\frac{1}{x_1}(u_m^-)^2dS \le M(1).
\end{align*}
Therefore
\begin{align*}
\int_{\D B_{t_m}^+}&\left( x_1(u_m^+)^2+\frac{1}{x_1}(u_m^-)^2\right)dS\int_{B_{t_m}^+}\left(x_1|\nabla v_m^+|^2+\frac{1}{x_1}|\nabla v_m^-|^2\right)dx  \\
&\le
t_m^{-1}\gamma \int_{\D B_{t_m}^+}x_1(u^+)^2dS+t_m^{-1}(\gamma+1)\int_{\D B_{t_m}^+}\frac{1}{x_1}(u^-)^2dS+t_m^{2\gamma+1}M(1).
\end{align*}
Consequently,
\begin{equation}\label{integr}
\begin{aligned}
\int_{B_{t_m}^+}\left(x_1|\nabla v_m^+|^2+\frac{1}{x_1}|\nabla v_m^-|^2\right)dx & \le \frac{2^{2\gamma+1}M(1)}{\int_{\D B_{t_m}^+}\left(x_1(u_m^+)^2+\frac{1}{x_1}(u_m^-)^2\right)dS}\\
&+ (\gamma+1).
\end{aligned}
\end{equation}

We claim that
\begin{equation}\label{claim}
\int_{B_{t_m}^+}\left(x_1|\nabla v_m^+|^2+\frac{1}{x_1}|\nabla v_m^-|^2\right)dx\le C \ \ \text{ and } \int_{B_1^+}\frac{1}{x_1}|\nabla v_m^-|^2dx\rightarrow 0.
\end{equation}

The first part of the claim follows from \eqref{cacc2} and \eqref{integr}.
Moreover, by \eqref{ACFconc} and \eqref{bound}, we have
\begin{equation}\label{quotient}
\frac{\int_{B_1^+}\frac{1}{x_1}|\nabla v_m^-|^2dx}{\int_{B_1^+}x_1|\nabla v_m^+|^2dx}\rightarrow 0.
\end{equation}

The second part of the claim follows by combining the first part of the claim, \eqref{bound} and\eqref{quotient}.

combining \eqref{bound} with \eqref{quotient}.

Let $t$ be the limit of a subsequence $(t_{m_k})_k$; since $(v_m)_m$ is bounded in $V^{1,2}(B_{t_m}^+)$, there exist $v_0$ such that $v_m \rightarrow v_0$ weakly in $V^{1,2}(B_t^+)$. Furthermore,  $v_0^-\equiv 0$ in $B_1^+$ and the compact embedding of $V^{1,2}(B_t^+)$ in $L^2_w(\D B_t^+)$ gives
\begin{equation}\label{integralSt}
\int_{\D B_{t}^+}\left(x_1(v_0^+)^2+\frac{1}{x_1}(v_0^-)^2\right)dS=1.
\end{equation}

We define the auxiliary functions in $B_t^+$
\[
z_m(x)=\begin{cases}
 x_1 v_m^+(x), \text{ if } v_m(x)>0 \\
v_m^-(x), \text{ if } v_m(x) <0.
\end{cases}
\]
Notice that on $\D\{z_m>0\}\cap\{x_1\neq 0\}$, $|\nabla z_m^+|=|\nabla z_m^-|$.
Since $\text{div}( x_1\nabla v_m)=0$ in $\{v_m>0\}$ and $v_m^+=\frac{1}{x_1}z_m^+$, in $\{z_m>0\}\cap B_t^+$ we have
\begin{align*}
0&=\text{div}(x_1\nabla v_m)=\text{div}(\nabla z_m)-\text{div}\left(\frac{e_1}{x_1}z_m\right)\\
&=\Delta z_m-\frac{\partial_1 z_m}{x_1}+\frac{z_m}{x_1^2}.
\end{align*}

Since in $\{v_m<0\}$ we have that
\[
0=\text{div}\left(\frac{1}{x_1}\nabla v_m\right)=\frac{1}{x_1}\Delta v_m-\frac{1}{x_1^{2}}\partial_1v_m,
\]
then in $\{z_m<0\}\cap B_t^+$ we have
\[
\Delta z_m=\frac{1}{x_1}\partial_1 v_m=\frac{1}{x_1}\partial_1z_m.
\]
We conclude that in $B_t^+$
\begin{equation}\label{w}
0=\Delta z_m-\frac{1}{ x_1}\partial_1z_m +x_1^{-2}z_m^+.
\end{equation}

Since $\xi_m\in  \D\{u_m>0\}$, then $\xi_m\in \D\{z_m>0\}$. Moreover, by hypothesis the free boundary of $u$ is assumed to be, in a neighborhood of the origin, a continuous injective curve, hence
\[
\mathcal{H}^{1,\infty}\left(\left\{x_1\ge \frac{\varepsilon}{2}\right\}\cap \{z_m=0\} \right)\ge \frac{\varepsilon}{4},
\]
where given $A\subset \R^2$,
\[
\mathcal{H}^{1,\infty}(A):=\frac{\omega_2}{2}\inf\Big\{\sum_{j=1}^{\infty}\text{diam}(S_j) \ : \ A\subset \cup_{j=1}^{\infty}S_j, \ \text{diam}(S_j) <\infty\Big\}.
\]

Moreover, up to a subsequence, there exists $z_0\in C^{1,\alpha}_{\text{loc}}(B_t)$ such that $z_m$ converges to $z_0$ in $W^{2,p}_{\text{loc}}(B_t^+)$ for each $1\le p<\infty$. Locally in $B_t^+$, $z_0$ solves
\[
0=\Delta z_0-\frac{1}{x_1}\partial_1z_0+x_1^{-2}z_0^+,
\]
and $|\nabla z_0^+|=|\nabla z_0^-|$ on $\D\{z_0>0\}\cap\{x_1> 0\}\cap B_t^+$.
As $v_0^-\equiv 0$ in $B_1^+$, then $z_0^-\equiv 0$ in $B_1^+$. Finally, since $z_m\rightarrow z_0$, then by Lemma 11.5 of \cite{giusti},
\begin{equation}\label{Haus}
\begin{aligned}
\mathcal{H}^{1,\infty}&\left(\left\{x_1\ge \frac{\varepsilon}{2}\right\}\cap\{z_0=0\} \cap B_1^+\right)\\
&\ge \limsup \mathcal{H}^{1,\infty}\left(\left\{x_1\ge \frac{\varepsilon}{2}\right\}\cap\{z_m=0\} \cap B_1^+\right)\\
&\ge \frac{\varepsilon}{4}.
\end{aligned}
\end{equation}
Since $z_0^-\equiv 0$ in $B_1^+$, we obtain that $|\nabla z_0^+|=0$ on $\D\{z_0>0\}\cap\{x_1> 0\}\cap B_1^+$, which implies $\D\{z_0>0\}\cap\{x_1> 0\}\cap B_1^+\subseteq \{z\in \{z_0=0\} \ | \ |\nabla z_0|=0\}$. Consequently, we conclude from \eqref{Haus} that
\begin{equation}\label{CF}
\mathcal{H}^{1,\infty}\left( \left\{x_1\ge \frac{\varepsilon}{2}\right\}\cap \{z_0=0, |\nabla z_0|=0\}\cap B_1^+\right)\ge \frac{\varepsilon}{4}.
\end{equation}
By Lemma 11.2 of \cite{giusti},
$\mathcal{H}^{1}\left( \left\{x_1\ge \frac{\varepsilon}{2}\right\}\cap \{z_0=0, |\nabla z_0|=0\}\cap B_1^+\right)>0$. Therefore, by Theorem 3.1 of \cite{cf1}, $z_0=0$ in $B_t^+$ and so $V^{1,2}(B_t^+)\ni v_0$ is zero in $B_t^+$, which contradicts \eqref{integralSt}.

\end{proof}

\section{Regularization away from $\{x_1=0\}$}
Assume $n=2$.
\begin{theorem}\label{T:exist} There exist $u \in C^{1,\alpha}(\overline{\{u>0\}}\cap \Omega\cap \{x_1>0\})\cap C^{1,\alpha}(\overline{\{u<0\}}\cap\Omega\cap \{x_1>0\})$ and a set $S$ consisting of mostly locally isolated points such that $u$ is a classical solution of \eqref{intro_sol} in $(\Omega\cap\{x_1>0\})\setminus S$ and the level set $\{u=0\}\cap \{x_1>0\}\cap\Omega$ is, outside $S$, locally a $C^{1,\alpha}$ graph.
\end{theorem}

\begin{proof} Assume formally that $u$ solves \eqref{intro_sol} and define the auxiliary function

\[
v(x)=
\begin{cases}
x_1u^+(x), \  \text{if } u(x)>0,\\
u^-(x), \  \text{if } u(x)<0.
\end{cases}
\]

Since $\text{div}(x_1\nabla u)=0$ in $\{u>0\}$ and $u^+=\frac{1}{x_1}v^+$, we obtain that in $\{v>0\}$
\begin{align*}
0&=\text{div}(x_1\nabla u)=\text{div}(\nabla v)-\text{div}\left(\frac{e_1}{x_1}v\right)\\
&=\Delta v-\frac{\partial_1 v}{x_1}+\frac{v}{x_1^2}.
\end{align*}

Since in $\{u<0\}$ we have
\[
0=\text{div}\left(\frac{1}{x_1}\nabla u\right)=\frac{1}{x_1}\Delta u-\frac{1}{x_1^{2}}\partial_1u,
\]
then in $\{v<0\}$ we have
\[
\Delta v=\frac{1}{x_1}\partial_1 u=\frac{1}{x_1}\partial_1v.
\]
We conclude that
\begin{equation}\label{v}
0=\Delta v-\frac{1}{x_1}\partial_1v +x_1^{-2}v^+.
\end{equation}

We consider a regularization of \eqref{v},
\begin{equation}\label{v2}
0=\Delta v_{\varepsilon}-\frac{1}{x_1}\partial_1v_{\varepsilon} +x_1^{-2}B_{\varepsilon}(v_{\varepsilon}).
\end{equation}
where $B_{\varepsilon}\in C^{\infty}(\R)$, $B_{\varepsilon}(z)\ge \chi_{\{z>0\}}:=B(z)$ in $\R$ and $B_{\varepsilon}\searrow B$ as $\varepsilon \rightarrow 0$.

There exists a maximal (and minimal) solution of \eqref{v2}, $v_{\varepsilon}$ for each $\varepsilon >0$, in the sense that $v_{\varepsilon} \in W^{2,p}(\Omega)$ for every $p\in [1,+\infty)$ and $v_{\varepsilon} \ge w$, for every subsolution $w\in W^{2,n}(\Omega)$ of \eqref{v2} in $\Omega'\subseteq \Omega$ such that $w\le v_{\varepsilon}$ on $\D \Omega'$.

By $W^{2,p}$ estimates, $(v_{\varepsilon})$ is bounded in $W^{2,p}_{\text{loc}}(\Omega\cap \{x_1>0\})$, for any $p\in [1,\infty)$, so there exists $v\in C^{1,\alpha}$ such that up to a subsequence, $v_{\varepsilon}$ converges to $v$. Moreover, $v$ is a solution of \eqref{v}.

Let $x_0\in \partial\{u>0\}\cap \{x_1> 0\}$ and $r>0$ be small enough so that $B_r(x_0)\subseteq \{x_1 >0\}$, hence $x_0\in\partial\{v>0\}\cap\{x_1> 0\}$. Denote with
\[
S=\{x\in B_r(x_0) \ | \ v(x)=0=|\nabla v(x)|\}=\{x\in B_r(x_0) \ | \ u(x)=0 = |\nabla u(x)|\}.
\]
By Theorem 3.1 of \cite{cf1}, we conclude that $S$ consists of isolated singularities and $(\{v=0\}\cap\{x_1> 0\}\cap B_r(x_0))\setminus S$ is $C^{1,\alpha}$.

We can now show existence of a classical solution $u$. Define
\[
u(x)=
\begin{cases}
\frac{v^+(x)}{x_1}, \text{ if } v(x)>0\\
v^-(x), \text{ if } v(x)<0.
\end{cases}
\]
We conclude that away from $\{x_1=0\}$ $u$ is $C^{1,\alpha}(\overline{\{u>0\}})\cap C^{1,\alpha}(\overline{\{u<0\}})$, and $(\{ u=0\}\cap \{x_1> 0\}\cap B_r(x_0))\setminus S$ is $C^{1,\alpha}$.
\end{proof}

\section{Compliance with Ethical Standards}
Funding: M. Smit Vega Garcia and G.S. Weiss have been partially supported by the project ``Singularities in ElectroHydroDynamic equations'' of the German Research Council.

Conflict of Interest: The authors declare that they have no conflict of interest.

\end{document}

%% file: axi1.pdf_t
\begin{picture}(0,0)%
\includegraphics{axi1.pdf}%
\end{picture}%
\setlength{\unitlength}{2368sp}%
\begingroup\makeatletter\ifx\SetFigFont\undefined%
\gdef\SetFigFont#1#2#3#4#5{%
  \reset@font\fontsize{#1}{#2pt}%
  \fontfamily{#3}\fontseries{#4}\fontshape{#5}%
  \selectfont}%
\fi\endgroup%
\begin{picture}(10834,4964)(-2411,-5183)
\put(3676,-2311){\makebox(0,0)[lb]{\smash{{\SetFigFont{11}{13.2}{\rmdefault}{\mddefault}{\updefault}{\color[rgb]{0,0,0}$\phi >0$}%
}}}}
\end{picture}%